\documentclass[11pt,a4paper]{article}
\usepackage{bbm}
\usepackage{CJK}
\usepackage{latexsym}
\usepackage{amssymb}
\usepackage{graphicx}
\usepackage{epsfig}
\usepackage{mathrsfs}
\usepackage{amssymb}
\usepackage{amsmath}
\usepackage[dvips]{color}
\usepackage{subfigure}
\usepackage{graphicx}
\usepackage{epsfig}
\usepackage{float}
\usepackage{caption}
\usepackage{times}
\usepackage{txfonts}
\usepackage{graphicx} %

\newtheorem{theorem}{Theorem}
\newtheorem{lemma}{Lemma}

\newtheorem{remark}{Remark}
\newenvironment{proof}{{\bf Proof:}}{\hfill$\bull$\medskip}
\newcommand{\bull}{\vrule height 1.8ex width 1.0ex depth 0ex}

\begin{document}
\begin{center}
{\bf\LARGE Error Bounds for Numerical Integration of Oscillatory Bessel
Transforms with Algebraic or Logarithmic
Singularities}\footnote{This work is supported by National Natural Science Foundation of China (Grant Nos.11301125, 11071260, 11226305,11301222), Scientific Research Startup Foundation of Hangzhou
Dianzi University (KYS075613017), the Fundamental Research Funds for
the Central Universities (No. 21612336) and NSF of Guangdong (No.
S2012040007860).}.

\vspace{0.4in}{\small Hongchao Kang\footnote{E-mail address: laokang834100@163.com}}

\vspace{0.1in}{\small 
Department of Mathematics, School of Science, Hangzhou Dianzi University, Hangzhou, Zhejiang 310018, PR China}

\vspace{0.4in} {\small Congpei An\footnote{E-mail address: tancpei@jnu.edu.cn,\,andbach@163.com}}

\vspace{0.1in}{\small Department of Mathematics, Jinan University,
Guangzhou 510632, China}

\end{center}

\noindent{\bf Abstract} In this paper, we present and analyze the Clenshaw-Curtis-Filon
methods for computing two classes of oscillatory Bessel transforms
with algebraic or logarithmic singularities. More importantly, for
these quadrature rules we derive new computational sharp error
bounds by rigorous proof. These new error bounds share the
advantageous property that some error bounds are optimal on $\omega$
for fixed $N$, while other error bounds are optimal on $N$ for fixed
$\omega$. Furthermore, we prove from the presented error bounds in
inverse powers of $\omega$ that the accuracy improves greatly, for
fixed $N$, as $\omega$ increases.

\vspace{0.05in} \noindent \textbf{Keywords:} oscillatory, Bessel transforms singularities,
Clenshaw-Curtis-Filon methods, quadrature rules error
bounds.

\vspace{0.05in} \noindent\\


\section{Introduction}
\setcounter{equation}{0} \setcounter{lemma}{0}
\setcounter{proposition}{0} \setcounter{corollary}{0}

\hspace{0.1cm}   Highly oscillatory Bessel transforms arise widely
in mathematical and numerical modeling of oscillatory phenomena in
many areas of sciences and engineering such as astronomy,
electromagnetics, acoustics, scattering problems, physical optics,
electrodynamics, and applied mathematics
\cite{Arfken,Bao1,Davies1,Huybrechs1}. In this paper, we focus on
new computational sharp error bounds of the quadrature rules for
singular oscillatory Bessel transforms of the forms
\begin{eqnarray}\label{1.1}
I_1[f]&=&\int_0^bx^\alpha f(x)J_m(\omega x)dx,\\
\label{1.2} I_2[f]&=&\int_0^bx^\alpha \ln(x)f(x)J_m(\omega x)dx,
\end{eqnarray}
where $f(x)$ is suitably smooth in $[0,b]$, $J_m(z)$ is the Bessel
function \cite{Abram} of the first kind and of order $m$ with
$Re(m)>-1$, $\omega$ is a large parameter, $b$ are real and finite,
and $\alpha>-1$. In particular, it should be noticed that transforms
\eqref{1.1} and \eqref{1.2} are integrals with algebraic and
logarithmic singularities, respectively.

\hspace{0.1cm} In most of the cases, such integrals cannot be
calculated analytically and one has to resort to numerical methods
\cite{Davis}. The numerical evaluation can be difficult when the
parameter $\omega$ is large, because in that case the integrand is
highly oscillatory. The singularities of algebraic or logarithmic
type and possible high oscillations of the integrands in (1.1) and
(1.2) make the above integrals very difficult to approximate
accurately using standard methods, e.g., \emph{Gaussian quadrature
rules}. It is well known \cite{Davis} that a prohibitively large
number of quadrature points is needed if one uses a classic rule
such as Gaussian quadratures, or any quadrature method based on
(piecewise) polynomial interpolation of the integrands.

In the last few decades, much progress has been made in developing
numerical schemes for generalized Bessel transform
$\int_a^bf(x)J_m(\omega g(x))dx$ without singularity. For example,
the modified Clenshaw-Curtis method \cite{Piessens1} was introduced
for efficiently computing $\int_0^1f(x)J_m(\omega x)dx$ for $m$
being an integer in 1983; the Levin method \cite{Levin}, Levin-type
method \cite{Olver}, and generalized quadrature rules
\cite{Evans,Xiang} were also available for approximating
$\int_a^bf(x)J_m(\omega x)dx$ for $Re(m)>-1$. However, \emph{the
Levin method}, \emph{Levin-type methods} and \emph{generalized
quadrature rules} cannot be used if $0 \not\in[a, b]$. In addition,
based on a diffeomorphism transformation, the reference
\cite{Xiang1} extended the Filon-type method to the efficient
computation of the integrals $\int_0^1f(x)J_m(\omega g(x))dx$ with
the exotic oscillator $g(x)$ satisfying that for $r\geq0$, and
$g(0)= g'(0)= \cdots = g^{(r)}(0) = 0, g^{(r+1)}(0)\neq0,
g'(x)\neq0$ for $x \in(0,1]$, where $Re(m) > {1}/({r+1})$. In many
situations the accuracy of the Filon-type method proposed in
\cite{Xiang1} is significantly higher than that of other methods. As
a matter of fact, it requires the solution of a linear system that
becomes more ill-conditioned as the number of interpolation nodes
increases, and one has to adopt higher-order digit arithmetic to get
the required accuracy. Furthermore, to avoid the Runge phenomenon,
the \emph{Clenshaw-Curtis-Filon-type method} \cite{Xiang2} based on
Clenshaw-Curtis points is designed for computing Bessel transform
$\int_a^bf(x)J_m(\omega x)dx$ without singularity. Here, it should
be also mentioned that the homotopy perturbation method in
\cite{Chen3,Chen4} was presented to compute $\int_a^bf(x)J_m(\omega
x)dx$. Recently, Chen \cite{Chen1}\cite{Chen2} also proposed two
different complex integration methods for approximating
$\int_a^bf(x)J_m(\omega x)dx$ if $0 \not\in[a, b]$. To the best of
our knowledge, so far little research has been done on the numerical
computation of the integrals (1.1) and (1.2) with an algebraic or
logarithmic singularity.

Consequently, our aim is to demonstrate high efficiency of the
proposed quadrature rules for such integrals (1.1) and (1.2) by
constructing error bounds. In the next section, we propose the
Clenshaw-Curtis-Filon methods for computing the integrals (1.1) and
(1.2). Here, the required modified moments can be efficiently
calculated by a recurrence relation. Section 3 sets up new and
computational sharp error bounds of these quadrature rules by theory
analysis. In Section 4, we design a higher order method and derive
its error estimate in inverse power of $\omega$. From these new
error bounds, it can be seen that for fixed $\omega$, the error
bounds are optimal on $N$, while for fixed $N$ the error bounds are
optimal on $\omega$. Moreover, for fixed $N$, the larger the values
of $\omega$, the higher the accuracy.

\section{\textbf{Clenshaw-Curtis-Filon methods for computing (1.1) and (1.2)}}

Chebyshev interpolation has precisely the same effect as taking
partial sum of an approximation Chebyshev series expansion
\cite{Mason}. Suppose that $f(x)$ is absolutely continuous on
$[0,b]$. Let $P_Nf(x)$ denote an interpolant of $f(x)$ of degree $N$
in the Clenshaw-Curtis points
\begin{equation}\label{2.2}
x_k=\frac{b}{2}+\frac{b}{2}\cos\left(\frac{k\pi}{N}\right),\quad
k=0,1,...,N.
\end{equation}
Then, the polynomial $P_Nf(x)$ can be expressed by (see [24, Eq.
6.27, 6.28])
\begin{equation}\label{2.3}
P_Nf(x)=\sum_{j=0}^N {''}a_jT^{**}_j(x),\ \text{where} \,\,
a_j=\frac{2}{N}\sum_{k=0}^N {''}f(x_k)T^{**}_j(x_k),
\end{equation}
where the double primes indicate that the first and last terms of
the sum are to be halved, $T_j^{**}(x)$ denotes the shifted
Chebyshev polynomial of the first kind of degree $j$ on $[0, b]$.
The coefficients ${a_j}$ can be computed efficiently by FFT
\cite{Dahlquist,Trefethen}.

The \emph{Clenshaw-Curtis-Filon} (CCF) methods for (1.1) and (1.2)
are defined, respectively, as follows,
\begin{eqnarray}
\nonumber I_1^{CCF}[f]&=&\int_0^bx^\alpha P_Nf(x)J_m(\omega
x)dx\\
\label{2.4}&=&b^{\alpha+1}\sum_{j=0}^N{''}a_jM_j,
\end{eqnarray}
and
\begin{eqnarray}
\nonumber I_2^{CCF}[f]&=&\int_0^bx^\alpha \ln(x)P_Nf(x)J_m(\omega x)dx\\
\label{2.5}&=&b^{\alpha+1}\sum_{j=0}^N{''}a_j[\ln(b)M_j+\widetilde{M}_j],
\end{eqnarray}
where, for $r=b\omega,$
\begin{eqnarray}
\label{2.6} M_j&=&\int_0^1x^\alpha T^*_j(x)J_m(r x)dx,\\
\label{2.7}\widetilde{M}_j&=&\int_0^1x^\alpha \ln(x)T^*_j(x)J_m(r
x)dx,
\end{eqnarray}
are called the modified moments, where $T_j^{*}(x)$ denotes the
shifted Chebyshev polynomial on $[0, 1]$, and which can be computed
efficiently, as described below.

\textbf{Fast computations of the modified moments:}

\indent{ }\normalsize The homogeneous recurrence relation of the
modified moments $M_j$, was provided by Piessens \cite{Kythe}
\cite{Piessens}, as follows:
\begin{eqnarray}\label{2.8}
&&\nonumber \frac{r^2}{16}M_{j+4}+[(j+3)(j+3+2\alpha)+\alpha^2-m^2-\frac{r^2}{4}]M_{j+2}\\
&&\nonumber+[4(m^2-\alpha^2)-2(j+2)(2\alpha-1)]M_{j+1}\\
&&\nonumber-[2(j^2-4)+6(m^2-\alpha^2)-2(2\alpha-1)-\frac{3r^2}{8}]M_{j}\\
&&\nonumber+[4(m^2-\alpha^2)-2(j-2)(2\alpha-1)]M_{j-1}\\
&&+[(j-3)(j-3-2\alpha)+\alpha^2-m^2-\frac{r^2}{4}]M_{j-2}+\frac{r^2}{16}M_{j-4}=0,
\end{eqnarray}
It is worth to notice that
$$\frac{\partial}{\partial\alpha}M_j=\widetilde{M}_j.$$
Therefore, by differentiating the above recurrence relation
\eqref{2.8} with respect to $\alpha$, we find $\widetilde{M}_j$
satisfying the following recurrence relation:
\begin{eqnarray}\label{2.9}
&&\nonumber \frac{r^2}{16}\widetilde{M}_{j+4}+[(j+3)(j+3+2\alpha)+\alpha^2-m^2-\frac{r^2}{4}]\widetilde{M}_{j+2}\\
&&\nonumber+[4(m^2-\alpha^2)-2(j+2)(2\alpha-1)]\widetilde{M}_{j+1}\\
&&\nonumber-[2(j^2-4)+6(m^2-\alpha^2)-2(2\alpha-1)-\frac{3r^2}{8}]\widetilde{M}_{j}\\
&&\nonumber+[4(m^2-\alpha^2)-2(j-2)(2\alpha-1)]\widetilde{M}_{j-1}\\
&&\nonumber+[(j-3)(j-3-2\alpha)+\alpha^2-m^2-\frac{r^2}{4}]\widetilde{M}_{j-2}+\frac{r^2}{16}\widetilde{M}_{j-4}\\
&&\nonumber=-2(\alpha+j+3)M_{j+2}+4(2\alpha+j+2)M_{j+1}+4(3\alpha+1)M_{j}\\
&&\ \ \ +4(2\alpha-j+2)M_{j-1}+2(j-\alpha-3)M_{j-2}.
\end{eqnarray}

Because of the symmetry of the recurrence relation of the Chebyshev
polynomials $T_j(x)$, it is convenient to get $T_{-j}(x)=T_j(x),
j=1,2,...,$ and, consequently $T^*_{-j}(x)=T^*_j(x),$ $M_{-j}=M_j$
and $\widetilde{M}_{-j}=\widetilde{M}_j.$ It can be verified easily
that both \eqref{2.8} and \eqref{2.9} are valid, not only for
$j\geq5$, but for all integers of $j$. Unfortunately, for
\eqref{2.8} and \eqref{2.9} both the forward recursion and the
backward recursion are asymptotically unstable
\cite{Kythe,Piessens}. Nevertheless, in practical applications the
instability is less pronounced if $\omega\geq 2j$. Practical
experiments demonstrate that $M_j$ and $\widetilde{M}_{j}$ can be
computed accurately using the forward recursion as long as
$\omega\geq 2j$. But for $\omega<2j$ the loss of significant figures
increases and recursion in the forward direction is no longer
applicable. In this case Lozier's algorithm \cite{Lozier} or
Oliver's algorithm \cite{Oliver} has to be used. This means that
both \eqref{2.8} and \eqref{2.9} have to be solved as a boundary
value problem with six initial values and two end values. The
solution of this boundary value problem requires the solution of a
linear system of equations having a band structure. The end value
can be estimated by the asymptotic expansions in \cite{Kythe} or can
be set equal to zero. The Lozier's algorithm incorporates a
numerical test for determining the optimum location of the endpoint,
when the end value is set to be zero. The advantage is that a
user-required accuracy is automatically obtained, without
computation of the asymptotic expansion. For details one can refer
to \cite{Kythe,Lozier,Oliver,Piessens}.
To start the recurrence relation with $k=0,1,2,3,\ldots$, we only
need $M_0, M_1, M_2$, and $M_3$. By plugging the shifted Chebyshev
polynomials $T^*_0(x)=1$, $T^*_1(x)=2x-1$, $T^*_2(x)=8x^2-8x+1$ and
$T^*_3(x)=32x^3-48x^2+18x-1$ on $[0,1]$ into \eqref{2.6}, we obtain
\begin{eqnarray}
\nonumber M_0&=&G(\omega,m,\alpha),\\
  \nonumber M_1&=&2G(\omega,m,\alpha+1)-M_0,\\
  \nonumber  M_2&=&8G(\omega,m,\alpha+2)-4M_1-3M_0,\\
\nonumber  M_3&=&32G(\omega,m,\alpha+3)-6M_2-15M_1-10M_0.
\end{eqnarray}
Then, it is apparent from the above equalities that \begin{eqnarray}
\label{2.111} \widetilde{M}_0&=&\frac{\partial}{\partial\alpha}G(\omega,m,\alpha),\\
\widetilde{M}_1&=&2\frac{\partial}{\partial\alpha}G(\omega,m,\alpha+1)-\widetilde{M}_0,\\
\widetilde{M}_2&=&8\frac{\partial}{\partial\alpha}G(\omega,m,\alpha+2)-4\widetilde{M}_1-3\widetilde{M}_0,\\
\label{2.114}\widetilde{M}_3&=&32\frac{\partial}{\partial\alpha}G(\omega,m,\alpha+3)-6\widetilde{M}_2-15\widetilde{M}_1-10\widetilde{M}_0,
\end{eqnarray}
where, from [1, p.480], [15, p.676] and [23, p.44], we find several
moments formulae as follows, for $\Re(m+\alpha)>-1,$
\begin{eqnarray}
\nonumber G(\omega,m,\alpha)&=&\int_0^1x^\alpha J_m(r x)dx\\
\label{2.18}&=&\frac{2^\alpha\Gamma(\frac{m+\alpha+1}{2})}{r^{\alpha+1}\Gamma(\frac{m-\alpha+1}{2})}+\frac{1}{r^\alpha}[(\alpha+m-1)J_m(r)S_{\alpha-1,m-1}(r)-J_{m-1}(r)S_{\alpha,m}(r)],\ \ \ \ \ \\
\label{2.19} G(\omega,m,\alpha)
&=&\frac{r^m}{2^m(\alpha+m+1)\Gamma(m+1)}{_1F_2}(\frac{\alpha+m+1}{2};\frac{\alpha+m+3}{2},m+1;-\frac{r^2}{4}),\\
\label{2.20} G(\omega,m,\alpha)
&=&\frac{\Gamma(\frac{m+\alpha+1}{2})}{r\Gamma(\frac{m-\alpha+1}{2})}\sum_{j=0}^\infty\frac{(m+2j+1)\Gamma(\frac{m-\alpha+1}{2}+j)}{\Gamma(\frac{m+\alpha+3}{2}+j)}J_{m+2j+1}(r),
\end{eqnarray}
where $S_{\mu,\nu}(z),\Gamma(z),{_1F_2}(\mu;\nu,\lambda;z)$ denote a
Lommel function of the second kind, the gamma function, a class of
generalized hypergeometric function, respectively. Moreover,
${_1F_2}(\mu;\nu,\lambda;z)$ converges for all $|z|$. From [32,
p.346], $S_{\mu,\nu}(z)$ can be expressed in terms of
${_1F_2}(\mu;\nu,\lambda;z)$, namely,
\begin{eqnarray}\label{2.21}
\nonumber
S_{\mu,\nu}(z)&=&\frac{z^{\mu+1}}{(\mu+\nu+1)(\mu-\nu+1)}{_1F_2}(1;\frac{\mu-\nu+3}{2},\frac{\mu+\nu+3}{2};-\frac{z^2}{4})\\
&&-\frac{2^{\mu-1}\Gamma(\frac{\mu+\nu+1}{2})}{\pi\Gamma(\frac{\nu-\mu}{2})}(J_\nu(z)-\cos(\pi(\mu-\nu)/2)Y_\nu(z)),
\end{eqnarray}
where $Y_\nu(z)$ is a Bessel function of the second kind of order
$\nu$. The right-hand sides of (2.11-2.14) involve the derivatives
of the generalized hypergeometric function with respect to the
parameter $\alpha$, which have been shown in \cite{Kang3}. The
required derivatives of the gamma function are also described in
terms of the Psi (Digamma) function $\psi_0(z)$, such as
\cite{Abram}
\begin{eqnarray}\label{2.221}\Gamma'(z)=\Gamma(z)\psi_0(z).\end{eqnarray}

The efficient implementation of the modified moments is based on the
fast computation of the Lommel functions $S_{\mu,\nu}(z)$ and the
hypergeometric function ${_1F_2}(\mu;\nu,\lambda;z)$. Excellent
references in this area are \cite{GTS,Watson}. Obviously, when
programming the proposed algorithm in a language like Matlab, we can
calculate the values of $\Gamma(z), J_m(z)$ and
${_1F_2}(\mu;\nu,\lambda;z)$ by invoking the biult-in functions `
gamma$(z)$', ` besselj$(m,z)$' and calling mfun(` hypergeom',
$[\mu], [\nu,\lambda ], z$) from Maple, respectively.

\textbf{The computation of} $S_{\mu,\nu}(z)$:

(1) For large $|z|$ and $|\arg z|<\pi,$ we can calculate efficiently
$S_{\mu,\nu}(z)$ by truncating the following asymptotic expansion
(see \cite[pp. 351-352]{Watson}) in inverse powers of $z$:
\begin{eqnarray}
\nonumber S_{\mu,\nu}(z)&=&z^{\mu-1}\left\{1-\frac{(\mu-1)^2-\nu^2}{z^2}+\frac{[(\mu-1)^2-\nu^2][(\mu-3)^2-\nu^2]}{z^4}-\ldots\right.\\
\nonumber&&\left.+(-1)^p\frac{[(\mu-1)^2-\nu^2]\ldots[(\mu-2p+1)^2-\nu^2]}{z^{2p}}\right\}+O(z^{\mu-2p-2}),
\end{eqnarray}

(2) For small $|z|$, we prefer to compute $S_{\mu,\nu}(z)$ using
\eqref{2.21}.

So, when $r=b\omega$ is large, such as $r\geq50,$ we prefer to
compute the moments using \eqref{2.18}. when $r=b\omega$ is small,
for example $r<50,$ the moments \eqref{2.20} are available. This may
be due to the property that $J_m(r)$ is a fast decreasing function
of $m$ when $m>r$. Practical experiments also demonstrate that
$J_m(r)$ can decrease to zero quite rapidly when $m$ is a little
larger than $r$. Fortunately, the moments \eqref{2.19} is available
for all $r$.

\section{Error bounds of the CCF methods \eqref{2.4} and \eqref{2.5}}

To obtain results that are absolutely reliable for numerical
computations, it is necessary to construct a upper bound for the
corresponding error. In the following, we will consider new and
computational error bounds. These new error bounds share that for
fixed $N$, the error bounds are optimal on $\omega$, while for fixed
$\omega$ the error bounds are optimal on $N$.

In the following, in order to derive these new error bounds in
inverse powers of $\omega$, we first give Lemmas 3.1 and 3.2.

\begin{lemma}
For every $t\in[0,b]$ $(b>0)$ and $\alpha>-1$, it is true that, for
$\omega\geq1$,
\begin{eqnarray}
\label{3.23}\int_0^tx^\alpha J_m(\omega x)dx&=&{\displaystyle\left\{
\begin{array}{ll}O(\frac{1}{\omega^{\alpha+1}}), & \textrm{if}\ -1<\alpha<0,
\\O(\frac{1}{\omega}), & \textrm{if}\ \
\alpha\geq0,\end{array}\right.}\\
\label{3.24}\int_0^tx^\alpha \ln (x) J_m(\omega
x)dx&=&{\displaystyle\left\{
\begin{array}{ll}O(\frac{1+\ln(\omega)}{\omega^{\alpha+1}}), & \textrm{if}\ -1<\alpha\leq0,
\\O(\frac{1}{\omega}), & \textrm{if}\ \
\alpha>0.\end{array}\right.}
\end{eqnarray}
\end{lemma}
\begin{proof} We divide our proof in three steps.

(1) For $-1<\alpha<0,$ setting $y=\omega x$ yields that
\begin{eqnarray}
\nonumber \int_0^tx^\alpha J_m(\omega
x)dx&=&\frac{1}{\omega^{\alpha+1}}\int_0^{\omega t}y^\alpha
J_m(y)dy,\\
\nonumber \int_0^tx^\alpha \ln (x) J_m(\omega
x)dx&=&\frac{1}{\omega^{\alpha+1}}\left[\int_0^{\omega t}y^\alpha
\ln(y) J_m(y)dy-\ln(\omega)\int_0^{\omega t}y^\alpha
J_m(y)dy\right].
\end{eqnarray}
Obviously, whether the integral upper limit $\omega t$ in the
right-side of the above two formulae is finite or not, by
convergence tests for improper integrals (\emph{Cauchy's test} or
\emph{Dirichelet's test}), we know that the resulting defect or
infinite integrals are convergent. It leads to the first identities
in \eqref{3.23} and \eqref{3.24}.

(2) For $\alpha=0,$ combining $\int_0^\infty J_m(t)dt=1$
\cite[p.486]{Abram} and the moments formula \eqref{2.18}, we have
$$ \int_0^t J_m(\omega
x)dx=\frac{1}{\omega}\int_0^{\omega
t}J_m(y)dy=O\bigg(\frac{1}{\omega}\bigg).$$

If $0<\omega t\leq1,$ from \cite[p.362]{Abram} and \cite{Frank}, we
have
\begin{eqnarray}\label{3.251}|J_m(x)|\leq1,m\geq0,x\in\Re.\end{eqnarray}
So, the first identity in \eqref{3.24} follows that
\begin{eqnarray}
\nonumber \left|\int_0^t \ln (x) J_m(\omega
x)dx\right|&=&\frac{1}{\omega}\left|\int_0^{\omega
t}\ln(\frac{y}{\omega})
J_m(y)dy\right|\\
  \nonumber&\leq&\frac{1}{\omega}\int_0^{\omega
t}|\ln(y)-\ln(\omega)||J_m(y)|dy\\
\nonumber&\leq&\frac{1}{\omega}\int_0^1 \big(-\ln(y)+\ln(\omega)
\big) dy\\
\label{3.25}&=&\frac{1+\ln(\omega)}{\omega}.
\end{eqnarray}
If $\omega t>1,$ from the proof of \eqref{3.25}, we then obtain
\begin{eqnarray}
\nonumber \left|\int_0^t \ln (x) J_m(\omega
x)dx\right|&=&\frac{1}{\omega}\left|\int_0^{\omega
t}\ln(\frac{y}{\omega})
J_m(y)dy\right|\\
  \nonumber&\leq&\frac{1}{\omega}\left|\int_0^1(\ln(y)-\ln(\omega))J_m(y)dy\right|+\frac{1}{\omega}\left|\int_1^{\omega t}(\ln(y)-\ln(\omega))J_m(y)dy\right|\\
\label{3.26}&\leq&\frac{1+\ln(\omega)}{\omega}+\frac{1}{\omega}\left|\int_1^{\omega
t}(\ln(y)-\ln(\omega))J_m(y)dy\right|.
\end{eqnarray}
Using the mean value theorem for integrals, we have
$$
\int_1^{\omega t}(\ln(y)-\ln(\omega))J_m(y)dy=-\ln(\omega)\int_1^\xi
J_m(y)dy+\ln(t)\int_\xi^{\omega t} J_m(y)dy,\ \text{for}\ 1\leq
\xi\leq\omega t.
$$
Then, it follows that
\begin{eqnarray}
\nonumber \left|\int_1^{\omega
t}(\ln(y)-\ln(\omega))J_m(y)dy\right|&\leq&\ln(\omega)\left|\int_1^\xi
J_m(y)dy\right|+|\ln(t)|\left|\int_\xi^{\omega t} J_m(y)dy\right|\\
\label{3.27}&=&O(1+\ln(\omega)),
\end{eqnarray}
which is due to the fact that both $\int_1^\xi J_m(y)dy$ and
$\int_\xi^{\omega t} J_m(y)dy$ converge by referring to the identity
$\int_0^\infty J_m(t)dt=1$ and the moments formula \eqref{2.18}.
Thus, combining \eqref{3.26} and \eqref{3.27} yields the first
identity in \eqref{3.24}.

(3) For $\alpha>0$, by integrating by parts and noting the
differential relation [1,pp.361,439]
\begin{eqnarray}\label{3.281}d[x^{m+1}J_{m+1}(\omega x)]=\omega x^{m+1}J_m(\omega x)dx,\end{eqnarray} and together with the first identities in \eqref{3.23} and \eqref{3.24},  we
find
\begin{eqnarray}
\nonumber \int_0^{t}x^\alpha J_m(\omega
x)dx&=&\frac{1}{\omega}\int_0^t
x^{\alpha-m-1}d[x^{m+1}J_{m+1}(\omega x)]\\
\nonumber&=&\frac{1}{\omega}\left[x^\alpha J_{m+1}(\omega
x)\big|_0^t-(\alpha-m-1)\int_0^t x^{\alpha-1}J_{m+1}(\omega
x)dx\right]\\
\nonumber&=&O\bigg(\frac{1}{\omega}\bigg).
\end{eqnarray}
Similarly, we obtain $$\int_0^{t}x^\alpha \ln(x) J_m(\omega
x)dx=O\bigg(\frac{1}{\omega}\bigg).$$ This completes the proof.
\end{proof}

From Lemma 3.1, we prove Lemma 3.2.
\begin{lemma}
For $f\in C[0,b], \alpha>-1$ and $\omega\geq1$, it is true that,
\begin{eqnarray}
\label{3.28}\left|\int_0^bt^\alpha f(t)J_m(\omega
t)dt\right|&\leq&C_1(\omega)(|f(b)|+\int_0^b|f'(t)|dt),\\
\label{3.29}\left|\int_0^bt^\alpha \ln(t)f(t)J_m(\omega
t)dt\right|&\leq&C_2(\omega)(|f(b)|+\int_0^b|f'(t)|dt),
\end{eqnarray}
where
\begin{eqnarray}
\nonumber C_1(\omega)&=&{\displaystyle\left\{
\begin{array}{ll}\frac{c_1}{\omega^{\alpha+1}}, & \textrm{if}\ -1<\alpha<0,
\\\frac{c_2}{\omega}, & \textrm{if}\ \
\alpha\geq0,\end{array}\right.}\\
\nonumber C_2(\omega)&=&{\displaystyle\left\{
\begin{array}{ll}\frac{c_3(1+\ln(\omega))}{\omega^{\alpha+1}}, & \textrm{if}\ -1<\alpha<0,
\\\frac{c_4}{\omega}, & \textrm{if}\ \
\alpha\geq0,\end{array}\right.}
\end{eqnarray}
and $c_k(k=1,2,3,4)$ are four constants independent of $\omega$ and
$f$.
\end{lemma}
\begin{proof} Setting $F(t)=\int_0^tx^\alpha J_m(\omega t)dt,t\in[0,b]$,
together with Lemma 3.1, we then have
\begin{eqnarray}
\nonumber C_1(\omega)=||F(t)||_\infty={\displaystyle\left\{
\begin{array}{ll}\frac{c_1}{\omega^{\alpha+1}}, & \textrm{if}\ -1<\alpha<0,
\\\frac{c_2}{\omega}, & \textrm{if}\ \
\alpha\geq0,\end{array}\right.}
\end{eqnarray}
These together implies that, by integrating by parts,
\begin{eqnarray}
\nonumber \left|\int_0^bt^\alpha f(t)J_m(\omega r
)dt\right|&=&\left|\int_0^b f(t)dF(t)\right|\\
\nonumber&=&\left|f(t)F(t)\big|_0^b-\int_0^b F(t)f'(t)dt\right|\\
\nonumber&\leq&|f(b)||F(b)|+\int_0^b |F(t)||f'(t)|dt\\
\nonumber&\leq&|f(b)|||F(t)||_\infty+\int_0^b |f'(t)|dt||F(t)||_\infty\\
\nonumber&=&C_1(\omega)(|f(b)|+\int_0^b|f'(t)|dt).
\end{eqnarray}
It is now obvious that the assertion \eqref{3.28} holds. The proof
of \eqref{3.29} can be completed by the method analogous to that
used above.
\end{proof}

Meanwhile, it should also be noted that the following Lemma 3.3 also
plays an important role in the construction of error bounds.
\begin{lemma} (see \cite{Xiang2}) Let $n$ be a nonnegative integer.
If $f$ is analytic with $|f(z)|\leq M$ in the region
$\mathscr{E}_{\rho}$ bounded by the ellipse with foci $\pm 1$ and
major and minor semi-axes whose lengths sum to $\rho>1,$ then for
$x\in[-1,1]$,
\begin{equation}
\|f^{(n)}(x)-P_{N}^{(n)}f(x)\|_\infty\leq
\frac{2M(N+1)^{2n}}{({\rho}^N-{\rho}^{-N})(2n-1)!!}\sum_{j=0}^n
\biggl(\frac{2\rho}{(\rho-1)^2}\biggr)^{n+1-j},
\end{equation}
where $(2n-1)!!=1\cdot3\cdot5 \cdots (2n-1)$ and $(-1)!!=1$.
\end{lemma}

Based on the above Lemmas 3.1-3.3, we derive error bounds in inverse
powers of $\omega$ in the following Theorems 3.1-3.2. For $f\in
C^2[0,b],$ the error bound of the CCF methods \eqref{2.4} is shown
as follows.
\begin{theorem}Assume that $f\in C^2[0,b].$ Then the absolute error
of the CCF methods \eqref{2.4} for $\omega\geq1, \alpha>-1$ and
every fixed $N$, satisfies
\begin{eqnarray}
\label{3.31}|I_1[f]-I_1^{CCF}[f]|\leq \min{\displaystyle\left\{
\begin{array}{lll}\frac{b^{\alpha+1}}{\alpha+1}||f(x)-P_Nf(x)||_\infty,\\
bC_1(\omega)||f'(x)-P'_Nf(x)||_\infty,\\
\frac{C_1(\omega)}{\omega}[|f'(b)-P'_Nf(b)|+b(1+\frac{3}{2}|\alpha-m-1|)||f''(x)-P''_Nf(x)||_\infty].
\end{array}\right\}}\ \ \
\end{eqnarray}
\end{theorem}
\begin{proof} From the definition of $P_Nf(x)$, it is obvious that
\begin{equation}\label{3.32}f(0)-P_Nf(0)=f(b)-P_Nf(b)=0.\end{equation}
In the following the proof will be split into three parts.

(1) For the first inequality in \eqref{3.31}, it follows at once
from \eqref{3.251} that
\begin{eqnarray}
\nonumber
|I_1[f]-I_1^{CCF}[f]|&=&\bigg|\int_0^bx^\alpha(f(x)-P_Nf(x))J_m(\omega
x)dx\bigg|\\
\nonumber&\leq&\int_0^b|x^\alpha(f(x)-P_Nf(x))J_m(\omega
x) |dx\\
\nonumber&\leq&\int_0^bx^\alpha dx||f(x)-P_Nf(x)||_\infty\\
\nonumber&=&\frac{b^{\alpha+1}}{\alpha+1}||f(x)-P_Nf(x)||_\infty.
\end{eqnarray}

(2) By using Lemma 3.2 and the identities \eqref{3.32}, the second
inequality in \eqref{3.31} follows that
\begin{eqnarray}
\nonumber
|I_1[f]-I_1^{CCF}[f]|&=&\bigg|\int_0^bx^\alpha(f(x)-P_Nf(x))J_m(\omega
x)dx\bigg|\\
\nonumber&\leq&C_1(\omega)(|f(b)-P_Nf(b)|+\int_0^b|f'(x)-P'_Nf(x)|dx)\\
\nonumber&\leq&bC_1(\omega)||f'(x)-P'_Nf(x)||_\infty.
\end{eqnarray}

(3) Since $f(x)-P_Nf(x)$ and $f'(x)-P'_Nf(x)$ can be expanded in
terms of {\it Maclaurin expansions}, as follows,
\begin{eqnarray}
\nonumber
f(x)-P_Nf(x)&=&f(0)-P_Nf(0)+(f'(0)-P'_Nf(0))x+\frac{f''(\eta_1)-P''_Nf(\eta_1)}{2}x^2\\
\nonumber&=&(f'(0)-P'_Nf(0))x+\frac{f''(\eta_1)-P''_Nf(\eta_1)}{2}x^2,\
\ 0<\eta_1<x,\\
\nonumber
f'(x)-P'_N(x)&=&f'(0)-P'_Nf(0)+(f''(\eta_2)-P''_Nf(\eta_2))x, \ \
0<\eta_2<x,
\end{eqnarray}
then we have
\begin{eqnarray}
\nonumber
\bigg|\bigg(\frac{f(x)-P_Nf(x)}{x}\bigg)'\bigg|&=&\bigg|\frac{x(f'(x)-P'_Nf(x))-(f(x)-P_Nf(x))}{x^2}\bigg|\\
\nonumber&=&\big|(f''(\eta_2)-P''_Nf(\eta_2))-\frac{1}{2}(f''(\eta_1)-P''_Nf(\eta_1))\big|\\
\nonumber&\leq&\big|f''(\eta_2)-P''_Nf(\eta_2)\big|+\frac{1}{2}\big|f''(\eta_1)-P''_Nf(\eta_1)\big|\\
\label{3.34}&\leq&\frac{3}{2}||f''(x)-P''_Nf(x)||_\infty.
\end{eqnarray}

By integrating by parts and using Lemma 3.2, together with
\eqref{3.281}, \eqref{3.32}, \eqref{3.34}, and due to the limit
$$\lim_{x\rightarrow0+}\frac{f(x)-P_Nf(x)}{x}=f'(0)-P'_Nf(0),$$ we
then obtain
\begin{eqnarray}
\nonumber
|I_1[f]-I_1^{CCF}[f]|&=&\bigg|\int_0^bx^{\alpha+1}\frac{f(x)-P_Nf(x)}{x}J_m(\omega
x)dx\bigg|\\
\nonumber&=&\frac{1}{\omega}\bigg|\int_0^bx^{\alpha-m}\frac{f(x)-P_Nf(x)}{x}d[x^{m+1}J_{m+1}(\omega
x)]\bigg| \\
\nonumber&=&\bigg|\frac{1}{\omega}x^{\alpha+1}\frac{f(x)-P_Nf(x)}{x}J_{m+1}(\omega
x)\big|_0^b\\
\nonumber&&-\frac{1}{\omega}\int_0^bx^{m+1}J_{m+1}(\omega
x)d\bigg[x^{\alpha-m}\frac{f(x)-P_Nf(x)}{x}\bigg]\bigg| \\
\nonumber&\leq&\frac{1}{\omega}\big|\alpha-m-1\big|\bigg|\int_0^bx^{\alpha}\frac{f(x)-P_Nf(x)}{x}J_{m+1}(\omega
x)dx\bigg|\\
\nonumber&&+\frac{1}{\omega}\bigg|\int_0^bx^\alpha[f'(x)-P'_Nf(x)]J_{m+1}(\omega
x)dx\bigg|\\
\nonumber&\leq&\bigg|\alpha-m-1\bigg|\frac{C_1(\omega)}{\omega}\bigg(\bigg|\frac{f(b)-P_Nf(b)}{b}\bigg|+\int_0^b\bigg|\bigg(\frac{f(x)-P_Nf(x)}{x}\bigg)'\bigg|dx\bigg)\\
\nonumber&&+\frac{C_1(\omega)}{\omega}\big(|f'(b)-P'_Nf(b)|+\int_0^b|(f'(x)-P'_Nf(x))'|dx\big)\\
\label{3.351}&\leq&\frac{C_1(\omega)}{\omega}[|f'(b)-P'_Nf(b)|+b(1+\frac{3}{2}|\alpha-m-1|)||f''(x)-P''_Nf(x)||_\infty].
\ \ \ \ \ \ \ \ \
\end{eqnarray}

We have thus proved the theorem.
\end{proof}

For $f\in C^3[0,b],$ the error bound of the CCF methods \eqref{2.5}
is presented as follows.
\begin{theorem}Assume that $f\in C^3[0,b].$ Then the absolute error
of the CCF methods \eqref{2.5} for $\omega\geq1,\alpha>-1$ and every
fixed $N$, satisfies
\begin{eqnarray}\label{3.35}
|I_2[f]-I_2^{CCF}[f]|\leq \min{\displaystyle\left\{
\begin{array}{lll}{\displaystyle\left\{\begin{array}{lll}\frac{b^{\alpha+1}(1-(\alpha+1)\ln
(b))}{(\alpha+1)^2}||f(x)-P_Nf(x)||_\infty,& \textrm{if}\ \ 0<b\leq 1,\\
\frac{2+b^{\alpha+1}((\alpha+1)\ln
(b)-1)}{(\alpha+1)^2}||f(x)-P_Nf(x)||_\infty,& \textrm{if}\ \ b>1,
\end{array}\right.}\\
bC_2(\omega)||f''(x)-P''_Nf(x)||_\infty,\\
\frac{3b}{2\omega}(C_1(\omega)+|\alpha-m|C_2(\omega))||f''(x)-P''_Nf(x)||_\infty\\
+\frac{C_2(\omega)}{\omega}[\frac{1}{b}|f'(b)-P'_Nf(b)|+\frac{7}{3}b||f'''(x)-P'''_Nf(x)||_\infty].
\end{array}\right\}}
\end{eqnarray}
\end{theorem}
\begin{proof}
(1) For the first inequalities in \eqref{3.35}, it follows that
\begin{eqnarray}
\nonumber |I_2[f]-I_2^{CCF}[f]|&=&\bigg|\int_0^bx^\alpha
\ln(x)(f(x)-P_Nf(x))J_m(\omega
x)dx\bigg|\\
\nonumber&\leq&\int_0^b|x^\alpha \ln(x)(f(x)-P_Nf(x))J_m(\omega
x) |dx\\
\nonumber&\leq&\int_0^b|x^\alpha \ln(x)|dx||f(x)-P_Nf(x)||_\infty\\
\nonumber&=&{\displaystyle\left\{\begin{array}{lll}\frac{b^{\alpha+1}(1-(\alpha+1)\ln
(b))}{(\alpha+1)^2}||f(x)-P_Nf(x)||_\infty,& \textrm{if}\ \ 0<b\leq 1,\\
\frac{2+b^{\alpha+1}((\alpha+1)\ln
(b)-1)}{(\alpha+1)^2}||f(x)-P_Nf(x)||_\infty,& \textrm{if}\ \ b>1.
\end{array}\right.}
\end{eqnarray}

(2) The second inequality in \eqref{3.35} can be proved by the same
method as employed in the proof of the second inequality in
\eqref{3.31}.

(3) Since $f(x)-P_Nf(x), f'(x)-P'_Nf(x)$ and $f''(x)-P''_Nf(x)$ can
be expanded in terms of {\it Maclaurin expansions}, as follows,
\begin{eqnarray}
\nonumber
f(x)-P_Nf(x)&=&f(0)-P_Nf(0)+(f'(0)-P'_Nf(0))x+\frac{f''(0)-P''_Nf(0)}{2}x^2+\frac{f'''(\xi_1)-P''_Nf(\xi_1)}{6}x^3\\
\nonumber&=&(f'(0)-P'_Nf(0))x+\frac{f''(0)-P''_Nf(0)}{2}x^2+\frac{f'''(\xi_1)-P''_Nf(\xi_1)}{6}x^3,\
\ 0<\xi_1<x,\\
\nonumber
f'(x)-P'_Nf(x)&=&f'(0)-P'_Nf(0)+(f''(0)-P''_Nf(0))x+\frac{f'''(\xi_2)-P''_Nf(\xi_2)}{2}x^2,
\ \ 0<\xi_2<x,\\
\nonumber
f''(x)-P''_Nf(x)&=&f''(0)-P''_Nf(0)+(f'''(\xi_3)-P'''_Nf(\xi_3))x, \
\ 0<\xi_3<x,
\end{eqnarray}
we then have
\begin{eqnarray}
\nonumber
\bigg|\bigg(\frac{f(x)-P_Nf(x)}{x}\bigg)''\bigg|&=&\bigg|\frac{x^2(f''(x)-P''_Nf(x))-2x(f'(x)-P'_Nf(x))+2(f(x)-P_Nf(x))}{x^3}\bigg|\\
\nonumber&=&\big|(f'''(\xi_3)-P'''_Nf(\xi_3))-(f'''(\xi_2)-P'''_Nf(\xi_2))-\frac{1}{3}(f'''(\xi_1)-P'''_Nf(\xi_1))\big|\\
\label{3.36}&\leq&\frac{7}{3}||f'''(x)-P'''_Nf(x)||_\infty.
\end{eqnarray}

By integrating by parts and using Lemma 3.2, together with
\eqref{3.281}, \eqref{3.32}, \eqref{3.36}, and noting that
$$\lim_{x\rightarrow0+}x^{\alpha+1}\ln(x)\frac{f(x)-P_Nf(x)}{x}=0,$$ using the same argument as in the proof of \eqref{3.351}, we
can easily carry out the proof of the third inequality in
\eqref{3.35}.

The proof of the theorem is now complete.
\end{proof}
\begin{remark}By transferring the integra interval $[0, b]$ into
$[-1,1]$, these estimates $||f(x)-P_Nf(x)||_\infty$,
$||f'(x)-P'_Nf(x)||_\infty,||f''(x)-P''_Nf(x)||_\infty,||f'''(x)-P'''_Nf(x)||_\infty$
in Theorems 3.1-3.2, have been given by Lemma 3.3.
\end{remark}

As shown in the above Theorems 3.1-3.2, for fixed $N$, we give these
error bounds in inverse powers of $\omega$. In the sequel, for fixed
$\omega$, we consider error bounds in inverse powers of $N$. Here,
to derive these error bounds, we first establish the following Lemma
3.4.

\begin{lemma} For every $j\geq1$ and fixed $\omega$, it is true that
\begin{eqnarray}
\label{3.38}\int_0^bx^\alpha T^{**}_j(x)J_m(\omega
x)dx&=&{\displaystyle\left\{
\begin{array}{ll}O(\frac{1}{j^{2\alpha+2}}), & \textrm{if}\ -1<\alpha<-\frac{1}{2},
\\O(\frac{1}{j^2}), & \textrm{if}\ \
\alpha\geq-\frac{1}{2},\end{array}\right.}\\
\label{3.39}\int_0^bx^\alpha \ln (x) T^{**}_j(x)J_m(\omega
x)dx&=&{\displaystyle\left\{
\begin{array}{ll}O(\frac{1+\ln(j)}{j^{2\alpha+2}}), & \textrm{if}\ -1<\alpha<-\frac{1}{2},
\\O(\frac{1}{j}), & \textrm{if}\ \
\alpha=-\frac{1}{2},\\
O(\frac{1}{j^2}), & \textrm{if}\ \ \alpha>-\frac{1}{2}.
\end{array}\right.}
\end{eqnarray}
\end{lemma}
\begin{proof} For $-1\leq
t\leq1$, three transformations $x=\frac{b}{2}+\frac{b}{2}t, $
$t=\cos\theta$ and $\theta=\pi-2\varphi$, yields that
\begin{equation}\label{3.40}
\int_0^bx^\alpha T^{**}_j(x)J_m(\omega
x)dx=(-1)^j2b^{\alpha+1}\int_0^{\frac{\pi}{2}}\cos(2j\varphi)\sin^{2\alpha+1}(\varphi)\cos(\varphi)J_m(b\omega
\sin^2(\varphi))d\varphi.
\end{equation}

(1) In the case of $\alpha\geq-\frac{1}{2}$: Based on differential
relations $\cos(2j\varphi)d\varphi=\frac{1}{2j}d\sin(2j\varphi)$ and
$\sin(2j\varphi)d\varphi=-\frac{1}{2j}d\cos(2j\varphi),$ we can
derive the second identity in \eqref{3.38} by integrating
\eqref{3.40} by parts twice.

(2) In the case of $-1<\alpha<-\frac{1}{2}$: Setting
$\varphi=\frac{u}{2j},$ we have

\begin{eqnarray}\label{3.41}
\nonumber &&\int_0^bx^\alpha T^{**}_j(x)J_m(\omega x)dx\\
&=&(-1)^j\frac{b^{\alpha+1}}{2^{2\alpha+1}}\frac{1}{j^{2\alpha+2}}\int_0^{j\pi}u^{2\alpha+1}\cos(u)\left(\frac{\sin(\frac{u}{2j})}{\frac{u}{2j}}\right)^{2\alpha+1}\cos(\frac{u}{2j})J_m(b\omega
\sin^2(\frac{u}{2j}))du.
\end{eqnarray}
Now that the right-side improper integral in \eqref{3.41} is
convergent, it is evident to see that the first identity in
\eqref{3.38} holds.

Similarly, according to the fact that \begin{eqnarray}\nonumber
&&\int_0^bx^\alpha\ln(x) T^{**}_j(x)J_m(\omega x)dx\\
\nonumber&=&(-1)^j4b^{\alpha+1}\int_0^{\frac{\pi}{2}}\cos(2j\varphi)\sin^{2\alpha+1}(\varphi)\ln(\sin(\varphi))\cos(\varphi)J_m(b\omega
\sin^2(\varphi))d\varphi\\
\nonumber&&+[\ln(b)+(2^\alpha-1)\ln2]\int_0^bx^\alpha
T^{**}_j(x)J_m(\omega x)dx,
\end{eqnarray}
and the logarithmic relation
$$\ln\bigg(\sin\bigg(\frac{u}{2j}\bigg)\bigg)=\ln\bigg(\frac{\sin(\frac{u}{2j})}{\frac{u}{2j}}\bigg)+\ln(2j)-\ln(u),$$
together with the assertion \eqref{3.38}, by the same procedure in
the proof of \eqref{3.38}, we then obtain the desired result
\eqref{3.39}. We have thus proved the lemma.
\end{proof}
\begin{remark}Using the asymptotic theory of Fourier integrals (see
Erd$\acute{e}$lyi\cite{Erdelyi1,Erdelyi}), Piessens
\cite{Kythe,Piessens} established this asymptotic expansion for
$j\rightarrow\infty$:
\begin{eqnarray}\label{3.421}
\nonumber &&\int_0^1x^\alpha T^*_j(x)J_m(\omega x)dx\\
\nonumber&=&2^{-\alpha-1}\int_{-1}^1(1+x)^\alpha T_j(x)J_m(\omega
(x+1)/2)dx\\
&\sim&-2^{-\alpha-2}J_j(\omega)j^{-2}+(-1)^j2^{-3m-3\alpha-2}\frac{\omega^m}{\Gamma(m+1)}\cos((\alpha+1)\pi)\Gamma(2\alpha+2)j^{-2\alpha-2m-2}.\
\ \ \ \
\end{eqnarray}
Then, by differentiating \eqref{3.421} with respect to $\alpha$ and
using \eqref{2.221}, we have
\begin{eqnarray}\label{3.422}
\nonumber &&\int_0^1x^\alpha\ln(x) T^*_j(x)J_m(\omega x)dx\\
\nonumber&\sim&2^{-\alpha-2}\ln(2)J_j(\omega)j^{-2}+(-1)^{j+1}3\ln(2)2^{-3m-3\alpha-2}\frac{\omega^m}{\Gamma(m+1)}\cos((\alpha+1)\pi)\Gamma(2\alpha+2)j^{-2\alpha-2m-2}\\
\nonumber &&+(-1)^{j+1}\pi2^{-3m-3\alpha-2}\frac{\omega^m}{\Gamma(m+1)}\sin((\alpha+1)\pi)\Gamma(2\alpha+2)j^{-2\alpha-2m-2}\\
\nonumber &&+(-1)^{j}2^{-3m-3\alpha-1}\frac{\omega^m}{\Gamma(m+1)}\cos((\alpha+1)\pi)\Gamma(2\alpha+2)\psi_0(2\alpha+2)j^{-2\alpha-2m-2}\\
&&+(-1)^{j+1}2^{-3m-3\alpha-1}\frac{\omega^m}{\Gamma(m+1)}\cos((\alpha+1)\pi)\Gamma(2\alpha+2)j^{-2\alpha-2m-2}\ln(j).
\end{eqnarray}
However, for each fixed $j$, the asymptotic expansions \eqref{3.421}
and \eqref{3.422} are divergent for $\omega$. Moreover, the
estimates in \eqref{3.38} and  \eqref{3.39} can not be directly
derived from these asymptotic formulas \eqref{3.421} and
\eqref{3.422} particularly for $\alpha>0$.
\end{remark}

With the aid of the above Lemma 3.4, we derive error bounds in
inverse powers of $N$ as follows.
\begin{theorem}Assume that $f(x)$ has an absolutely continuous
$(k-1)$st derivative $f^{(k-1)}(x)$ on $[0,b]$ and a $k$th
derivative $f^{(k)}(x)$ of bounded variation $V_k$ for some
$k\geq1.$ Then, for every fixed $\omega$, and $N\geq k,$ the
absolute errors of the CCF methods \eqref{2.4} and \eqref{2.5}
satisfy
\begin{eqnarray}
\label{3.42}|I_1[f]-I_1^{CCF}[f]|&=&{\displaystyle\left\{
\begin{array}{ll}O(\frac{1}{N^{k+2\alpha+2}}), & \textrm{if}\ -1<\alpha<-\frac{1}{2},
\\O(\frac{1}{N^{k+1}}), & \textrm{if}\ \
\alpha\geq-\frac{1}{2},\end{array}\right.}\\
\label{3.43}|I_2[f]-I_2^{CCF}[f]|&=&{\displaystyle\left\{
\begin{array}{ll}O(\frac{\ln(N)}{N^{k+2\alpha+2}}), & \textrm{if}\ -1<\alpha<-\frac{1}{2},\\
O(\frac{\ln(N)}{N^{k+1}}), & \textrm{if}\ \alpha=-\frac{1}{2},\\
O(\frac{1}{N^{k+1}}), & \textrm{if}\ \ \alpha>-\frac{1}{2}.
\end{array}\right.}
\end{eqnarray}
\end{theorem}
\begin{proof}
Recalling that $T_n^{**}(x)=T_n(\frac{2x}{b}-1)$ and \eqref{2.2}, we
have
\begin{eqnarray}
\nonumber&&T_{2pN\pm j}^{**}(x_k)=T_{2pN\pm
j}\bigg(\frac{2x_k}{b}-1\bigg)=T_{2pN\pm
j}\bigg(\cos\bigg(\frac{k\pi}{N}\bigg)\bigg)=\cos\bigg((2pN\pm
j)\frac{k\pi}{N}\bigg)\\
\label{3.46}&&=\cos\bigg(\frac{jk\pi}{N}\bigg)=T_j\bigg(\cos\bigg(\frac{k\pi}{N}\bigg)\bigg)=T_{j}\bigg(\frac{2x_k}{b}-1\bigg)=T_{j}^{**}(x_k),
\end{eqnarray}
for each $j$, $k=0,1,\ldots,N$ and $p=1,2,\ldots$.

From \eqref{2.3} and \eqref{3.46}, together with the discrete
orthogonality of Chebyshev polynomials (see \cite[Section
4.6]{Mason}), we obtain
\begin{eqnarray}\label{3.47}
P_NT_{pN+j}^{**}(x)&=&{\displaystyle\left\{
\begin{array}{ll}T_{N-j}^{**}(x), & \textrm{if $p$ is odd},
\\T_{j}^{**}(x), & \textrm{if $p$ is even}.\end{array}\right.}
\end{eqnarray}
Then, we can see directly from \eqref{2.4}, \eqref{2.5} and
\eqref{3.47} that
\begin{eqnarray}
\label{3.48}I_1^{CCF}[T_{pN+j}^{**}(x)]&=&{\displaystyle\left\{
\begin{array}{ll}I_1[T_{N-j}^{**}(x)], & \textrm{if $p$ is odd},
\\I_1[T_{j}^{**}(x)], & \textrm{if $p$ is
even},\end{array}\right.}\\
\label{3.49}I_2^{CCF}[T_{pN+j}^{**}(x)]&=&{\displaystyle\left\{
\begin{array}{ll}I_2[T_{N-j}^{**}(x)], & \textrm{if $p$ is odd},
\\I_2[T_{j}^{**}(x)], & \textrm{if $p$ is even}.\end{array}\right.}
\end{eqnarray}

Therefore, we can deduce from Lemma 3.4, \eqref{3.48} and
\eqref{3.49}, that the sums of aliasing errors for the CCF methods
\eqref{2.4} and \eqref{2.5} can be estimated for $p$ being a
positive integer by
\begin{eqnarray}
\nonumber\sum_{j=0}^NI_1^{CCF}[T_{pN+j}^{**}(x)]&=&{\displaystyle\left\{
\begin{array}{ll}\sum_{j=0}^NI_1[T_{N-j}^{**}(x)]=\sum_{j=0}^N\int_0^bx^\alpha T^{**}_{N-j}(x)J_m(\omega
x)dx, & \textrm{if $p$ is odd},
\\\sum_{j=0}^NI_1[T_{j}^{**}(x)]=\sum_{j=0}^N\int_0^bx^\alpha T^{**}_{0}(x)J_m(\omega
x)dx, & \textrm{if $p$ is
even},\end{array}\right.}\\
\label{3.50}&=&{\displaystyle\left\{\begin{array}{ll}O(\frac{1}{N^{2\alpha+1}}),
& \textrm{if}\ -1<\alpha<-\frac{1}{2},
\\O(1), & \textrm{if}\ \
\alpha\geq-\frac{1}{2},\end{array}\right\} \textrm{$p$ is either odd or even;}}\\
\nonumber\sum_{j=0}^NI_2^{CCF}[T_{pN+j}^{**}(x)]&=&{\displaystyle\left\{
\begin{array}{ll}\sum_{j=0}^NI_2[T_{N-j}^{**}(x)]=\sum_{j=0}^N\int_0^bx^\alpha\ln(x) T^{**}_{N-j}(x)J_m(\omega
x)dx, & \textrm{if $p$ is odd},
\\\sum_{j=0}^NI_2[T_{j}^{**}(x)]=\sum_{j=0}^N\int_0^bx^\alpha \ln(x)T^{**}_{j}(x)J_m(\omega
x)dx, & \textrm{if $p$ is
even},\end{array}\right.}\\
\label{3.51}&=&{\displaystyle\left\{\begin{array}{ll}O(\frac{\ln(N)}{N^{2\alpha+1}}),
& \textrm{if}\ -1<\alpha<-\frac{1}{2},\\
O(\ln(N)), & \textrm{if}\ \alpha=-\frac{1}{2},
\\O(1), & \textrm{if}\ \
\alpha>-\frac{1}{2},\end{array}\right\} \textrm{$p$ is either odd or
even.}}
\end{eqnarray}

Moreover, we can find that Theorem 4.2 in \cite{Trefethen} implies
the estimate
\begin{eqnarray}
\label{3.491}|a_j|=O\bigg(\frac{1}{j^{k+1}}\bigg).
\end{eqnarray}

Combining \eqref{3.50}, \eqref{3.51} and \eqref{3.491}, and by a
similar way as shown in the proof of Theorem 5.1 in
\cite{Trefethen}, we can deduce the desired results \eqref{3.42} and
\eqref{3.43}.
\end{proof}

\section{Extension to a higher order method and  error estimate}
The choice of the extreme points as interpolation points for highly
oscillatory integrals is not only a technical necessity but also can
improve the accuracy significantly \cite{Olver,Xiang,Xiang2}.
Moreover, only by adding derivative information of $f(x)$ at the
endpoints $0$ and $b$ can the asymptotic order of the method be
improved \cite{Kang1,Kang,Olver,Xiang,Xiang2}. By the
above-mentioned particular observations, the higher order CCF
methods for (1.1) and (1.2) can be defined, respectively, by
\begin{eqnarray}
\nonumber I_1^{HCCF}[f]&=&\int_0^bx^\alpha P_{N+2s}f(x)J_m(\omega
x)dx\\
\label{4.55}&=&b^{\alpha+1}\sum_{j=0}^{N+2s}b_jM_j,\\
\nonumber I_2^{HCCF}[f]&=&\int_0^bx^\alpha \ln(x)P_{N+2s}f(x)J_m(\omega x)dx\\
\label{4.56}&=&b^{\alpha+1}\sum_{j=0}^{N+2s}b_j[\ln(b)M_j+\widetilde{M}_j],
\end{eqnarray}
where $P_{N+2s}f(x)$ is the special \emph{Hermite interpolation
polynomial} \cite{Kang1,Xiang2} of $f(x)$ at the Clenshaw-Curtis
points $c_n =\frac{b}{2}+\frac{b}{2}\cos\left(\frac{n\pi}{N}\right)$
on $[0,b]$ satisfying
\begin{equation}\label{4.57}
P_{N+2s}^{(k)}f(0)=f^{(k)}(0),\quad P_{N+2s}f(c_n)=f(c_n),\quad
P_{N+2s}^{(k)}f(b)=f^{(k)}(b),
\end{equation}
for $n=1,2,...,N-1$ and $k=0,1,...,s;$  it can be evaluated in
$O(N\log N)$ operations \cite{Xiang2}. Moreover, the polynomial
$P_{N+2s}f(x)$ can be expressed by
\begin{equation}
P_{N+2s}f(x)=\sum_{j=0}^{N+2s} b_jT_j^{**}(x),
\end{equation}
where the coefficients ${b_j}$ can be computed efficiently by an
algorithm \cite{Xiang2}.

Here, we establish error bounds in inverse powers of $N$ as follows.
\begin{theorem}Assume that $f\in C^{s+2}[0,b].$ Then the absolute
errors of the higher order CCF methods \eqref{4.55} and \eqref{4.56}
for $\alpha>-1$ and every fixed $N$, satisfy
\begin{eqnarray}
|I_1[f]-I_1^{HCCF}[f]|&=&O\bigg(\frac{C_1(\omega)}{\omega^{s+1}}\bigg),\\
|I_2[f]-I_2^{HCCF}[f]|&=&O\bigg(\frac{C_2(\omega)}{\omega^{s+1}}\bigg).
\end{eqnarray}
\end{theorem}
\begin{proof} Based on the differential relation \eqref{3.281}, we
have
\begin{eqnarray}
\nonumber I_1[f]-I_1^{HCCF}[f]&=&\int_0^bx^\alpha(f(x)-P_{N+2s}f(x))J_m(\omega x)dx\\
\label{4.61}&=&\frac{1}{\omega}\int_0^bx^{\alpha+s-m}\frac{f(x)-P_{N+2s}f(x)}{x^{s+1}}d[x^{m+1}J_{m+1}(\omega
x)],\\
\nonumber
I_2[f]-I_2^{HCCF}[f]&=&\int_0^bx^\alpha\ln(x)(f(x)-P_{N+2s}f(x))J_m(\omega
x)dx\\
\label{4.62}&=&\frac{1}{\omega}\int_0^bx^{\alpha+s-m}\ln(x)\frac{f(x)-P_{N+2s}f(x)}{x^{s+1}}d[x^{m+1}J_{m+1}(\omega
x)].
\end{eqnarray}
By using \eqref{4.57} and resorting to integrating \eqref{4.61} and
\eqref{4.62} by parts $s+1$-time, respectively, together with Lemma
3.2, we establish the desired results as in the proof of the third
inequalities in Theorems 3.1-3.2.
\end{proof}

\section{Conclusion}
This paper presents a series of new sharp error bounds of the
Clenshaw-Curtis-Filon methods for two classes of oscillatory Bessel
transforms with algebraic or logarithmic singularities. From the
above error bounds, it is worth noting that the required accuracy
can be obtained by using derivatives of $f(x)$ at the end-points or
adding the number of the interior node points. Moreover, the
Clenshaw-Curtis-Filon methods posses the advantageous property that
for fixed $N$ the accuracy increases when oscillation becomes
faster, which can be also obtained directly from these error bounds
in inverse powers of $\omega$.







%

\end{document}